\documentclass[12pt]{amsart}

\usepackage{srcltx} 

\usepackage[ansinew]{inputenc}
\usepackage[english]{babel}
\usepackage{anysize}
\usepackage{amsfonts}
\usepackage{amssymb}
\usepackage{amsmath}
\usepackage{amsthm}
\usepackage{doc}
\usepackage{exscale}
\usepackage{fontenc}
\usepackage{ifthen}
\usepackage{latexsym}
\usepackage{graphicx}
\usepackage{float}
\usepackage{array}
\usepackage[all]{xy}
\usepackage[usenames,dvipsnames]{color}
\usepackage{hyperref}
\usepackage{enumerate}
\usepackage{multirow}
\usepackage{color}
\usepackage{microtype}
\usepackage{url}
\usepackage{varioref}

\marginsize{2.5cm}{2.5cm}{2.5cm}{2.5cm}

\newtheorem{thm}{Theorem}[section]
\newtheorem{cor}[thm]{Corollary}

\newtheorem{rmk}[thm]{Remark}

\newtheorem*{thm*}{Theorem}

\newcommand{\caI}{\mathcal{I}}

\newcommand{\C}{\mathbb{C}}

\newcommand{\PP}{\mathbb{P}}
\newcommand{\OO}{\mathcal{O}}

\DeclareMathOperator{\cod}{cod}

\DeclareMathOperator{\codim}{codim}

\DeclareMathOperator{\rk}{rk}

\title[On the Xiao conjecture for plane curves]{On the Xiao conjecture for plane curves}

\thanks{The first named author has been partially supported by INdAM - GNSAGA and by FIRB2012 "Moduli spaces and applications", which are granted by MIUR, and by IMUB during his stay at the University of Barcelona; the second named author has been partially supported by the Proyecto de Investigaci\'on MTM2015-65361-P; the third named author is partially supported by PRIN 2015 
\emph{``Moduli spaces and Lie Theory''},  INdAM - GNSAGA and FAR 2016 
(PV) \emph{``Variet\`a algebriche, calcolo algebrico, grafi orientati e 
topologici''}. We would like to thank Miguel {\'A}ngel Barja, V{\'i}ctor Gonz{\'a}lez-Alonso and Carlos d'Andrea for useful comments.}

\author{F. F. Favale}
\address[Filippo F. Favale]{Department of Mathematics, 
Universit\`a degli Studi di Trento, via Sommarive 14,
38123 Trento, Italy}
\email{filippo.favale@unitn.it}

\author{J. C. Naranjo}
\address[Juan Carlos Naranjo]{Departament de Matem\`atiques i Inform\`atica, 
Universitat de Barcelona, Gran Via 585, 
08007 Barcelona, Spain}
\email{jcnaranjo@ub.edu}

\author{G. P. Pirola}
\address[Gian Pietro Pirola]{
Department of Mathematics, Universit\`a degli Studi di Pavia,
via Ferrata 1, 27100 Pavia, Italia}
\email{gianpietro.pirola@unipv.it}

\subjclass[2010]{14H50,14D06,14J99,14B10}

\begin{document}

\begin{abstract}
Let $f: S\longrightarrow B$ be a non-trivial fibration from a complex projective smooth surface $S$ to a smooth curve $B$ of genus $b$. Let $c_f$ the Clifford index of the general fibre $F$ of $f$. In \cite{BGN} it is proved that the relative irregularity of $f$, $q_f=h^{1,0}(S)-b$ is less or equal than or equal to $g(F)-c_f$. In particular this proves the (modified) Xiao's conjecture: $q_f\le \frac {g(F)}2 +1$ for fibrations of general Clifford index. In this short note we assume that the general fiber of $f$ is a plane curve of degree $d\ge 5$ and we prove that $q_f\le g(F)-c_f-1$. In particular we obtain the conjecture for families of quintic plane curves. This theorem is implied for the following result on infinitesimal deformations: let $F$ a smooth plane curve of degree $d\ge 5$ and let $\xi$ be an infinitesimal deformation of $F$ preserving the planarity of the curve. Then the rank  of the cup-product map $H^0(F,\omega_F) {\overset{ \cdot \xi} \longrightarrow} H^1(F,O_F)$ is at least $d-3$. We also show that this bound is sharp.       
\end{abstract}

\maketitle


\section{Introduction}

\noindent Let $f: S\longrightarrow B$ be a non-isotrivial fibration from a complex projective smooth surface $S$ to a smooth curve $B$ of genus $b$. A natural question is trying to understand the relation between the invariants of the surface, the base curve $B$ and of the general fibre $F$. Non-isotrivial means that the smooth fibres are not isomorphic to each other, in other words: the natural modular map $B^0\longrightarrow \mathcal M_g$ in the moduli space of curves of genus $g=g(F)$ is not constant, where $B^0$ is the open set of $B$ with smooth fibres.     
The invariants we consider are the relative irregularity $0\le q_f:=h^{1,0}(S)-g(B)$ and the genus $g$ of the general fibre. Xiao proved in \cite{Xiao1} that for non-isotrivial fibrations the inequality
\[
 q_f\le \frac {5g+1}6
\]
holds, 
and he formulated in \cite{Xiao2} the conjecture
\[
 q_f \le \frac{g+1}2.
\]
After some counterexamples found by Pirola in \cite{Pir} (see also \cite{AP}) the conjecture is nowadays reformulated as follows

\noindent {\bf Modified Xiao's conjecture:} For non-isotrivial fibrations it holds that
\[
 q_f \le \frac{g}2+1
\]
Observe that it is equivalent to the initial conjecture for $g$ odd.
\vspace{2mm}

\noindent There are some evidences for this conjecture: Xiao proved that it holds if $b=0$ and it is known to be true when $F$ is hyperelliptic (see \cite{Cai}). Moreover in \cite{BGN} a upper bound of $q_f$ is found in terms of the Clifford index $c_f$ of the general fibre. Remember that the Clifford index of the curve $F$ is defined as:
\[
 \text{Cliff}(F):=\min \{deg (L) -2(h^0(F,\OO_F(L))-1) \mid h^0(F,\OO_F(L))\ge2, h^1(F,\OO_F(L))\ge 2\}. 
\]
Clifford's Theorem states that $\text{Cliff}(F)\ge 0$ and it is $0$ if and only if $F$ is hyperelliptic. It is known that $\text{Cliff}(F)=1$ if and only if either $F$ is trigonal or isomorphic to a smooth quintic plane curve. A smooth plane curve of degree $d$ has Clifford index $d-4$ and a general curve in $\mathcal M_g$ has Clifford index $\lfloor \frac {g-1}2 \rfloor$.  
\vspace{2mm}

\noindent The main theorem in \cite{BGN} says that for a non-isotrivial fibration it holds
\[
 q_f\le g-c_f.
\]
In particular the conjecture is true when the general fiber has general Clifford index. Combining the results given above  one easily checks that the conjecture is true for $g\le 4$ and that the first open case corresponds to $g\ge 5$ and $c_f=1$. In this paper we take care of the quintic plane curve case. More generally we consider families of smooth plane curves of degree $d\ge 5$. Our main theorem is the following:

\begin{thm}\label{main1}
 Let $f:S\longrightarrow B$ be a fibration such that the general fibre is a plane curve of degree $d\ge 5$. Then
 \[
  q_f \le g-c_f-1=g-(d-4)-1=g-d+3
 \]
\end{thm}

\noindent In the case $d=5$, hence $g=6$, we obtain $q_f \le 4$ which is the predicted bound. Hence we obtain:
\begin{cor}
The modified Xiao's conjecture holds for fibrations with general fibre a quintic plane curve.
\end{cor}

\noindent The idea of the proof of (\ref{main1}) is as follows: let us fix a general point of $B$ and let $F$ the fibre at this point, then $f$ induces an infinitesimal deformation $\xi \in H^1(F,T_F)$. The kernel $W_{\xi }$ of the cup-product map		\noindent -
\[		
H^0(F,\omega_F) {\overset{ \cdot \xi} \longrightarrow} H^1(F,\OO_F)		
\]		
contains the vector space $H^0(S,\Omega^1_S)/f^*H^0(B,\omega_B)$ (see \cite[Section 2]{BGN} for the details). Therefore $q_f \le \dim W_{\xi}=g-rank (\cdot \xi)$. Thus it is enough to find a lower bound of the rank of the map given by the cup-product with $\xi$. Then the next Theorem immediately implies Theorem (\ref{main1}):		
\begin{thm}\label{main2}		
 Let $\xi$ be an infinitesimal deformation of $F$ as smooth plane curve, then the rank of the map  $H^0(F,\omega_F) {\overset{ \cdot \xi} \longrightarrow} H^1(F,\OO_F)$ is at least $d-3$, and this bound is realized for the Fermat curve.		
\end{thm}		
\noindent The rest of the paper is devoted to the proof of Theorem (\ref{main2}). In the next section we recall how to use the Jacobian ring of a plane curve in order to understand the cohomology of the curve $F$ and the cup-product maps. We also state two theorems of Green on multiplication and restriction maps of polynomials. In section $3$ we use these facts combined with the classical Macaulay's Theorem to prove Theorem (\ref{main2}).


\section{Preliminaries}

\noindent Let $F$ be a smooth curve of genus $g\geq 3$ and an let  $\xi\in H^1(F,T_F)$ be a non-trivial infinitesimal deformation of $F$. To simplify notations we call 
{\it rank} of $\xi$ to the rank of the cup-product map
\[
H^0(F,\omega_F) {\overset{ \cdot \xi} \longrightarrow} H^1(F,\mathcal O_F)
\]
considered in the Introduction. We keep the notation $W_{\xi}$ for the kernel of this map.

\noindent Let $F$ be a smooth plane curve defined as the zero locus of the homogeneous polynomial $f\in \C[x,y,z]=S$ of degree $d$. Its {\it Jacobian ring} $R$ is the graded ring 
\[
 R=\bigoplus_{n\geq 0}R^n=\bigoplus_{n\geq 0}\left(S^n/J^n\right).
\]
Here $J^n$ is the degree $n$ part of the {\it Jacobian ideal} $J=(f_x,f_y,f_z)$. As $F$ is smooth, $\{f_x,f_y,f_z\}$ is a regular sequence and this implies that $R$ satisfies the following properties:

\begin{thm}[Macaulay]
\label{THM:MACAU}
Let $N$ be $3(d-2)$. Then $R^N$ has dimension $1$ and, for every $k$ such that $0\leq k\leq N$ we have that the multiplication map $$R^{k}\otimes R^{N-k}\rightarrow R^N$$ 
is a perfect pairing. Moreover $R^k=0$ for $k>N$ or $k<0$ and the dimension of $R^k$ for $0\leq k\leq N$ is determined only by $d$.
\end{thm}

\noindent In addition to these, Griffiths proved that one can read canonically several pieces of the Hodge structure of $F$ in $R$. More precisely

\begin{thm}
\label{THM:GRIFF}
Let $R$ be the Jacobian ring of a smooth plane curve of degree $d$. Then
\begin{itemize}
\item $H^0(F,\omega_F)\simeq S^{d-3}=R^{d-3}$;
\item $H^1(F,\OO_F)\simeq R^{N-d+3}=R^{2d-3}$;
\item the subspace of $H^1(F,T_F)$ of all the infinitesimal deformations that preserve the planarity of $F$ is isomorphic to $R^d$;
\item multiplication in $R$ induces, using the previous identifications, cup product of the corresponding elements. 
\end{itemize}
\end{thm}

\noindent We refer to \cite{Voisin} for a proof of these facts.
\vspace{2mm}

\noindent Next we quote two theorems of Green concerning properties of $S=\mathbb C[x,y,z]$. We stress that they are valid in any dimension although we state (and use) them only in dimension $n=2$. The first theorem can be found in \cite[Lecture 7, page 74]{Green_CIME}, putting $p=0$ :

\begin{thm}[Green]
\label{THM:GREEN}
Let $W$ be a subspace of $S^a$ of codimension $c$ and assume that $|W|$ is a base point free linear system on $\PP^2$. Then, for any $m\geq c$ one has that the multiplication map
$$W\otimes S^m\rightarrow S^{m+a}$$
is surjective.
\end{thm}

\noindent In order to state the second theorem of Green we need some notation.
Given a positive integer $a$, for any $c\ge 0$ there is a unique expression
$$
c=\binom {k_a}{a}+\binom {k_{a-1}}{a-1}+\dots +\binom {k_1}{1}=\binom {k_a}{a}+\binom {k_{a-1}}{a-1}+\dots +\binom {k_\delta}{\delta},
$$
such that $k_a > k_{a-1} >\dots >k_{\delta}\geq \delta >0.$  The numbers $(k_a,k_{a-1},\dots,k_\delta)$ are uniquely identified by this definition and are called {\it Macaulay's Coefficients of $c$ with respect to $a$}. 
\vspace{2mm}

\noindent If $(k_a,k_{a-1},\dots,k_{\delta})$ are the Macaulay coefficients with respect to $a$, we denote by $c_{\left\langle a \right\rangle}$ the number
$$
c_{\left\langle a \right\rangle}=\binom {k_a-1}{a}+\binom {k_{a-1}-1}{a-1}+\dots +\binom {k_\delta-1}{\delta},
$$
where $\binom{m}{n}=0$ if $m<n$. Keeping this terminology we can state the following theorem:

\begin{thm}[Green, \cite{Green_hyperplane} ]
\label{THM:GREENHYP}
Let $W\subset S^a$ be a linear system with codimension $c$. Let $H$ be a general line in $\mathbb P^2$. Then the codimension $c_H$ of the image of the restriction map
$$W \longrightarrow H^0(H,\OO_H(a))$$
satisfies $c_H\le c_{\left\langle a \right\rangle}$.
\end{thm}

\begin{cor}
\label{COR:GREENHYP}
Under the assumption of Theorem \ref{THM:GREENHYP}, if $c < a$ then $c_H=0$, i.e. the restriction map $W\rightarrow H^0(H,\OO_H(a))$ is surjective.
\end{cor}

\begin{proof}
If $c\leq a$ we can write $c=a-r$ for some $r$ such that $0\leq r\leq a$.
As 
$$c=a-r=\binom{a}{a}+\binom {a-1}{a-1}+\dots +\binom {r+1}{r+1}$$
we have that $(a,a-1,\dots,r+1)$ are the Macaulay's coefficients of $c$ with respect to $a$. Therefore 
$$c_{\left\langle a \right\rangle}=\binom{a-1}{a}+\binom {a-2}{a-1}+\dots +\binom {r}{r+1}=0.$$
By definition of $c_H$ we have that the restriction is surjective as claimed.
\end{proof}


\section{Proof of the Theorem (\ref{main2})}

\noindent The Theorem (\ref{main2}) in the introduction asserts that given a smooth curve $F$ of degree $d$, the rank of a non trivial infinitesimal deformation of $F$ which preserves the planarity of $F$ is bounded below by $d-3$. By using the identifications provided by Griffiths' results in (\ref{THM:GRIFF}) this translates into the following statement: 

\begin{thm}
\label{THM:MAIN}
Let $F$ be a smooth plane curve of degree $d\geq 5$ and let $R$ be its Jacobian ring. Let $\xi\in R^d\setminus \{0\}$, then the rank of the map
\[
 S^{d-3}=R^ {d-3} {\overset{ \cdot \xi} \longrightarrow} R^{2d-3}
\]
is at least $d-3$. 
\end{thm}

\begin{proof}
Denote by $W=W_{\xi}\subset S^{d-3}$ the kernel of the multiplication map $$\cdot\xi:S^{d-3}\rightarrow R^{2d-3}.$$
Note that the codimension of $W$ in $S^{d-3}$ is equal to the rank of $\xi$.
The proof is divided in two cases depending of the existence or not of a fixed loci.

\vskip 3mm
\noindent \textbf{Case 1: $|W|$ is base-point-free}.\newline
We are going to see that $ \rk(\xi) \ge d-2$. Assume that the opposite is true, i.e. that $\rk(\xi)\leq d-3$ holds. Then, by Green's Theorem, for every $m\geq \cod_{S^{d-3}}(W)=\rk(\xi)$ we have that the multiplication map
$$\mu_{m}:W \otimes S^{m}\rightarrow S^{m+d-3}$$
is surjective.
In particular, as $\rk(\xi)\leq d-3$ we can take $m=d-3$. Hence we have
$$\mu_{d-3}:W \otimes S^{d-3}\rightarrow S^{2d-6}$$
is surjective and the same holds for the map obtained by passing to the quotient, i.e. to the Jacobian ring:
$$\mu_{d-3}:W \otimes R^{d-3}\rightarrow R^{2d-6}.$$
But this is impossible since by the definition of $W$, the image of $\mu_{d-3}$ is killed by $\xi\neq 0$, hence the pairing 
$$R^{2d-6}\otimes R^d\rightarrow R^{3d-6}=R^N$$
is degenerated contradicting Macaulay's Theorem. Hence we have necessarily $\rk(\xi)\geq d-2$ as claimed.

\vspace{3mm}
\noindent \textbf{Case 2: $|W|$ is not base-point-free}.\newline
First we observe that we can assume that there are no base components. Indeed, assume that there exists a curve
$C$ of degree $0 < d'<d-3$  in the fixed part of $|W|$. Then we have that $W\subset C\cdot S^{d-d'-3}$
and therefore $\dim W \le \dim S^{d-d'-3}\le \dim S^{d-4}$, hence
\[
 \text{codim}\, W\geq \binom {d-1}2 - \binom {d-2}2=d-2,
\]
as wanted.

\noindent Hence now on we assume that $|W|$ has only isolated base points. We proceed by contradiction, so assume that $\rk(\xi)\leq d-4$ holds. 
\vspace{2mm}

\noindent Let $Z$ be the base locus of $|W |$ and denote by $\caI_Z$ the ideal sheaf of $Z$ as subscheme of $\PP^2$. Then the evaluation induces a surjection
$$W \otimes \OO_{\PP^2}\twoheadrightarrow \caI_Z(d-3).$$
Denoting by $M_W$ its kernel we have the short exact sequence
\begin{equation}
\label{SES:T1}
\xymatrix{
0\ar[r]& M_{W }\ar[r]& W \otimes \OO_{\PP^2}\ar[r]^-{ev}& \caI_Z(d-3)\ar[r]& 0.
}
\end{equation}
Let $s$ be a general element in $S^1=H^0(\PP^2,\OO_{\PP^2}(1))$. As the base points are isolated, we can assume that the line $L=\{s=0\}$ is disjoint with $Z$. By considering the multiplication by $s$, the short exact sequence   (\ref{SES:T1}) induces the following commutative diagram with exact rows and columns 
\begin{equation}
\label{SES:T2}
\xymatrix{
 &
	0\ar[d] &
	0\ar[d] &
	0\ar[d] &
	\\
0 \ar[r] &
    M_{W }\ar[r]\ar[d]_{\cdot s}&
    W \otimes \OO_{\PP^2}\ar[r]^-{ev}\ar[d]_{\cdot s}& 
    \caI_Z(d-3)\ar[r]\ar[d]_{\cdot s}&
    0\\
0\ar[r]&
	M_{W }(1)\ar[r]\ar[d]&
    W \otimes \OO_{\PP^2}(1)\ar[r]^-{ev_1}\ar[d]&
    \caI_Z(d-2)\ar[r]\ar[d]&
    0\\
0\ar[r]&
	M_L\ar[r]\ar[d]&
    W \otimes \OO_{L}(1)\ar[r]^-{ev_L}\ar[d]&
    \OO_L(d-2)\ar[r]\ar[d]&
    0.\\
&
	0 &
	0 &
	0 &
}
\end{equation}
The sheaf  $M_L$ in the third row is by the definiton the kernel of $ev_L$. Notice that the rightmost sheaf of the last row is $\OO_L(d-2)$ because we are assuming that $L$ is disjoint from $Z$.
\vspace{2mm}

\noindent {\bf Claim:} The vanishing $H^1(L,M_L)=0$ holds.
\begin{proof} (of the claim).
Indeed, under our hypothesis, Corollary \ref{COR:GREENHYP} implies that the restriction map
$$W \rightarrow H^0(L,\OO_L(d-3))$$
is surjective. Tensoring with $H^0(L,\OO_L(1))$ we get that also the evaluation map 
$$
ev_l: W \otimes H^0(L,\OO_L(1)) \longrightarrow H^0(L,\OO_L(d-3))\otimes H^0(L,\OO_L(1)) \twoheadrightarrow   H^0(L,\OO_L(d-2))
$$
is surjective. Then the cohomology sequence of the third row in the diagram gives the vanishing. 
\end{proof}

\noindent By taking cohomology in the second row we have a map: 
$$
\xymatrix{
W \otimes S^1\ar[r]^-{\eta_1}&H^0(\PP^2,\caI_Z(d-2))\ar[r]& H^1(\PP^2,M_{W }(1))\ar[r] & 0
}$$
and denote with $W_1$ the image of $\eta_1$. Our strategy is to replace $W$ by $W_1\subset S^{d-2}$ and apply again the same argument in order to finally reach a subspace of $W_2\subset S^{d-1}$ where it is easier to finish the proof as we will see below. Hence we need to show that the  codimension of $W_1$ in $S^{d-2}$ is lower than or equal to the codimension of $W$ in $S^{d-3}$. To prove this we consider the cohomology exact sequences of the first two rows of the diagram (\ref{SES:T2}):

\[
\xymatrix{
W \ar@{^{(}->}[r]\ar[d]_-{\cdot s}&
   H^0(\PP^2,\caI_Z (d-3))  \ar[r]\ar[d]_-{\cdot s} & 
    H^1(\PP^2,M_W)\ar[r]\ar@{->>}[d]_-{\cdot s}&
    0
    \\
W \otimes S^1\ar[r]^-{\eta_1} &
   H^0(\PP^2,\caI_Z(d-2))  \ar[r] & 
    H^1(\PP^2,M_W(1))\ar[r]&
    0
   }
\]
The last vertical arrow is surjective due to the claim. Therefore 
\begin{equation}\label{CODIM}
 \text{codim}_{H^0(\PP^2,\caI_Z(d-2))} W_1 = h^1(\PP^2,M_W(1))\le h^1(\PP^2,M_W)= \text{codim}_{H^0(\PP^2,\caI_Z(d-3))} W.
\end{equation}

\noindent Now we need to compare $\codim_{S^{d-2}} H^0(\PP^2,\caI_Z(d-2))$ with $\cod_{S^{d-3}}H^0(\PP^2,\caI_Z(d-3))$.
As we will observe in a moment, they coincide as a consequence of the vanishing of $H^1(L,M_L)=0$.
\vspace{2mm}

\noindent {\bf Claim:} We have the equality:
\[
 \text{codim}_{S^{d-2}}H^0(\PP^2,\caI_Z(d-2)) = \text{codim}_{S^{d-3}}H^0(\PP^2,\caI_Z(d-3)).
\]
\begin{proof} (of the claim).
We start again with the diagram (\ref{SES:T2}). The cohomology exact sequence of the last two columns gives
\[
\xymatrix{
W \ar[r]\ar[d]_{\cdot s}&
   H^0(\PP^2,\caI_Z (d-3)) \ar[d]_{\cdot s} 
    \\
W \otimes S^1\ar[r]^-{\eta_1} \ar[d]_{res_1} &
   H^0(\PP^2,\caI_Z(d-2)) \ar[d]_{res_2} 
   \\
W\otimes H^0(L,\mathcal O_L(1)) \ar[d] \ar[r]^{\alpha} & H^0(L,\mathcal O_L(d-2)) \ar[d] 
   \\
 0 \ar[r] & H^1(\PP^2,\caI_Z(d-3)).  
   }
\]
Since $H^1(L,M_L)=0$ both $\alpha $ and $\alpha \circ res_1$ are surjective. Therefore $res_2$ is also surjective. This implies the isomorphism
\[
 H^1(\mathbb P^2, \caI_Z(d-3)) {\overset{ \cong } \longrightarrow} H^1(\mathbb P^2, \caI_Z(d-2)).
\]
Now consider the diagram:
\[
 \xymatrix{
 0 \ar[r]& \mathcal I_Z(d-3) \ar[r] \ar[d] & \mathcal O_{\mathbb P^2}(d-3) \ar[r] \ar[d] & \mathcal O_Z(d-3) \ar[r] \ar[d]^{\cong } & 0 
    \\
 0 \ar[r] &\mathcal I_Z(d-2) \ar[r]& \mathcal O_{\mathbb P^2}(d-2) \ar[r]        & \mathcal O_Z(d-2) \ar[r]  & 0.
 }
\]
Implementing the isomorphism above we obtain the claim.
\end{proof}

\noindent Combining the inequality (\ref{CODIM}) with the claim we have:
$$\cod(W_1)\leq \cod(W )\leq d-4.$$

\noindent Observe that $|W_1|$ has, by construction, the same base locus of $|W |$ and, as we have just proven, $\cod(W_1)\leq  d-4$. Hence, we can apply the same argument using $W_1$ instead of $W $ starting with the surjection
$$W_1\otimes \OO_{\PP^2}\rightarrow \caI_Z(d-2).$$
By doing this we obtain $W_2\subset S^{d-1}$ which satisfy 
$$\cod(W_2)\leq \cod(W_1)\leq \cod(W )\leq d-4.$$
To finish the proof we consider the subspace
$$\tilde{W}:=W_2 + \langle f_x,f_y,f_z\rangle=W_2 + J^{d-1}.$$
Observe that in $\tilde{W}$ there is a section which doesn't vanish on at least one point (we are assuming that $Z$ is not empty) of the base locus, as the partial derivatives cannot all vanish in a point. Hence the dimension of $\tilde{W}$ has increased at least by $1$ and
$$\cod(\tilde{W})\leq \cod(W_2)-1\leq d-5.$$
Moreover $\tilde{W}$ is base point free by construction so we can apply Green's Theorem (\ref{THM:GREEN}) again with $m=d-5$ and we obtain that the multiplication map
$$\tilde{W}\otimes S^{d-5}\rightarrow S^{d-5+d-1}=S^{2d-6}$$
is surjective.
\vspace{2mm}

\noindent Passing to the quotient by the Jacobian ideal we have a surjective map
$$W_2\otimes R^{d-5}\rightarrow R^{2d-6}.$$ This implies by the definition of $W_2$ (image of $W_1 \otimes S^1$) and $W_1$ (image of $W\otimes S^1$) that all the elements in $R^{2d-6}$ are orthogonal to $\xi$
which contradicts Macaulay's Theorem. This finishes the prove of the bound.
\vspace{2mm}

\noindent Finally we see that the bound is sharp. Fix $d\geq 4$ and let $F$ be the Fermat curve of degree $d$ in $\PP^2$, i.e. the zero locus of the polynomial $f=x^d+y^d+z^d$.
In this case, the Jacobian ideal is simply 
\[
 J=(x^{d-1},y^{d-1},z^{d-1})
\]
and one can easily prove that $R^N=\langle (xyz)^{d-2}\rangle$.
Consider the element $[x^{d-2}y^2]\in R^d$ and denote it by $\xi$. Observe, moreover, that this is not zero as $x^{d-2}y^2\not\in J$.
If $m=x^ay^bz^c$ with $a+b+c=d-3$, we have $m\cdot \xi\neq 0$ if and only if $a= 0$ and $0\leq b\leq d-4$. Hence the image of the multiplication by $\xi$ is generated by the $d-3$ elements of
$$\{[x^{d-2}y^{2+a}z^{d-3-a}]\,|\, 0\leq a\leq d-4\}.$$
As they are independent we have $\rk(\xi)=d-3$. Notice that $W $ is generated by monomials and $x^{d-3}, y^{d-3}\in W $ but $z^{d-3}\not\in W $, hence the base locus of $|W |$ is $P_2=(0:0:1)$.
\end{proof}

\begin{rmk}
In Theorem \ref{THM:MAIN} we have proven that the bound 
$\rk(\cdot \xi)\geq d-3$ is sharp by showing that for the Fermat curve of degree $d$ there exists an element in $R^d$ whose rank is exactly $d-3$. There are other curves that have this property. For example, denote by $f_{\lambda}\in S^5$ the polynomial $x^5+y^5+z^5+5\lambda x^3y^2$ with $\lambda\in \C$ and consider the quintic $F_{\lambda}=\{f_\lambda=0\}$. Notice that $F_{0}$ is the Fermat quintic so the general quintic of this type is indeed smooth. Moreover, as $x^3y^2$ is not zero in the Jacobian ideal of the Fermat quintic, we know that $F_\lambda$ is not biholomorphic to $F_0$ for $\lambda$ general. If we denote by $R_\lambda$ the Jacobian ideal of $F_{\lambda}$ it is easy to see that $\xi_{\lambda}=[x^3y^2]$ is not zero in $R_\lambda$. Now we want to prove that $\rk(\xi_\lambda)=2$ for all $\lambda\in \C$. As we have already seen, $\xi_\lambda$ is non trivial so, by Theorem \ref{main2} we know that $\rk(\xi_\lambda)\geq 2$. In order to see that the equality holds it is enough to prove that $\dim W_{\xi_\lambda}\geq 4$ (as $\dim R_{\lambda}^2=6$). With a little bit of effort one can show that
$$x^2,xy,y^2,xz + 18\lambda^3yz \in  W_{\xi_{\lambda}}$$
thus proving the claim.
\end{rmk}

\end{document}